\documentclass[preprint,12pt]{elsarticle}

\usepackage[T1]{fontenc}
\usepackage[utf8]{inputenc}
\usepackage{graphicx}
\graphicspath{{figures/}}
\usepackage{amsmath,amssymb,amsthm}
\biboptions{sort&compress,numbers}

\newtheorem{theorem}{Theorem}[section]
\newtheorem{proposition}[theorem]{Proposition}
\newtheorem{lemma}[theorem]{Lemma}
\newtheorem{corollary}[theorem]{Corollary}

\newtheorem{remark}[theorem]{Remark}

\newcommand{\qbinom}[3]{\left[\begin{array}{c}#1\\#2\end{array}\right]_{#3}}
\newcommand{\Dyck}{\mathcal D}

\newcommand{\R}{\mathcal R}
\newcommand{\bbZ}{\mathbb Z}
\newcommand{\dd}{\mathrm d}

\journal{Advances in Applied Mathematics}

\begin{document}

\begin{frontmatter}

\title{Area and water-capacity statistics for upper hulls of Dyck paths}

\author[melbourne]{Aleksander L. Owczarek}
\ead{aleks.owczarek@gmail.com}
\author[qmul]{Thomas Prellberg\corref{cor1}}
\ead{t.prellberg@qmul.ac.uk}
\cortext[cor1]{Corresponding author.}

\address[melbourne]{School of Mathematics and Statistics, University of Melbourne, Victoria 3010, Australia}
\address[qmul]{School of Mathematical Sciences, Queen Mary University of London, London E1 4NS, United Kingdom}

\begin{abstract}
We study Dyck paths refined simultaneously by proper area and water capacity, where water capacity is measured above the path and below its lattice-path upper hull.  The finite-height ingredients used in the enumeration are classical bounded-height area-polynomial and continued-fraction objects.  The upper-hull decomposition produces a coupled area--capacity substitution, which gives an exact four-variable height expansion with denominator branches indexed by the height levels.  The full generating function is asymmetric in the two weights, while the height summands admit a symmetric unreduced denominator representation under interchange of the area and capacity weights.  In the open square $0<p,q<1$, we prove that the length radius of $G(x,1,p,q)$ is the minimum of the positive real denominator branches.  The proof combines uniform normal convergence of the height expansion below this first branch with a Perron-root representation of the branch locations and an interval log-submodularity theorem for spectral radii of weighted paths.  On the diagonal $p=q=s$, the classical Chebyshev specialisation gives explicit branch crossings and a $(1-s)^{2/3}$ branch-envelope accumulation law at the Dyck critical point.
\end{abstract}

\begin{keyword}
Dyck paths \sep water capacity \sep area statistics \sep Perron eigenvalues \sep log-concavity
\MSC[2020] 05A15 \sep 05A19 \sep 05A30 \sep 15A18 \sep 82B20
\end{keyword}

\end{frontmatter}

\setlength{\emergencystretch}{2em}

\section{Introduction}
\label{sec:introduction}

Dyck paths are a central class of Catalan objects and a standard testing ground for exact enumeration with geometric weights.  Their area-weighted enumeration belongs to the classical $q$-Catalan tradition \cite{carlitz1964,furlinger1985}, and leads naturally to finite continued fractions and height-restricted $q$-orthogonal forms \cite{flajolet1980,owczarek2012}.  Related asymptotic phenomena for area-weighted paths and lattice-vesicle models are treated in \cite{brak1994,prellberg1995,haug2015,haug2017}, with further vesicle variants in \cite{owczarek2010,owczarek2014}.  The bounded-height continued-fraction, denominator-polynomial and orthogonal-polynomial constructions are used below as standard background.  The new enumeration lies in a two-statistic refinement involving both the area below a Dyck path and the amount of water retained above it and below its upper hull.  Here ``upper hull'' means the lattice-path water-filling envelope defined precisely in Section~\ref{sec:model}, rather than the ordinary Euclidean convex hull.

Another prominent two-statistic refinement of Dyck paths is the $q,t$-Catalan theory, where area is paired with bounce or with diagonal inversions, in connection with diagonal harmonics and Macdonald-polynomial theory \cite{haglund2003,haglund2005,haglund2008}.  The present statistic is different in both origin and purpose.  Water capacity is defined from the upper-hull, or water-filling, gaps of a Dyck path; it differs from bounce and dinv, and the present problem is exact enumeration for two geometric weights rather than a $q,t$-symmetric refinement.  The problem here is to determine the exact length generating function refined simultaneously by proper area and retained water.

The water-capacity statistic has been studied in several enumerative settings.  It appears for words and their bargraph representations \cite{blecher2018words,archibald2021}, for integer compositions \cite{blecher2018compositions}, for permutations \cite{blecher2019permutations}, and for random functions \cite{blecher2025randomfunctions}.  Mansour and Shattuck studied water cells in bargraphs of compositions and set partitions \cite{mansour2018watercells}, as well as in pattern-restricted compositions \cite{mansour2019pattern}.  Recent work continues this direction for restricted compositions \cite{hopkins2025,shattuck2025} and for black or black--white cell-capacity statistics on words and Catalan polyominoes \cite{fried2025,baril2026}.

For Dyck paths, the direct predecessor on the water-capacity side is the work of Blecher, Brennan and Knopfmacher \cite{blecher2020}, who introduced the water capacity of Dyck paths, derived generating functions by summing over the maximum height, and obtained asymptotics for the average capacity at fixed semi-length.  Knopfmacher and Blecher subsequently treated the corresponding statistic on Dyck bridges, or Grand-Dyck paths, in \cite{blecher2025bridges}.  Related Catalan-word and bargraph results include work of Callan, Mansour and Ram\'irez on semiperimeter, area and exterior corners \cite{callan2021}, work of Mansour, Ram\'irez and Toquica on lattice-point and vertex-degree statistics \cite{mansour2021latticepoints}, and work of Mansour and Shattuck on Catalan and smooth words according to capacity \cite{mansour2025catalan}.

We keep the Dyck-path geometry of \cite{blecher2020} but weight two complementary geometric statistics at once.  The weight $q$ marks the proper area below the path, while $p$ marks the water capacity above the path and below the upper hull.  For a fixed hull, these two statistics partition the plaquettes below the hull into path-area plaquettes and retained-water plaquettes.  The resulting model formally contains the ordinary Catalan case $p=q=1$, the area-weighted line $p=1$, the capacity-weighted line $q=1$, and the diagonal hull-area specialisation $p=q$.  The exact generating-function identity specialises to all of these regimes.  The main analytic theorem of the paper has a narrower scope: it concerns the length radius of $G(x,1,p,q)$ in the open square $0<p,q<1$.  Boundary cases such as $p=1$, $q=1$, and $p=q=1$ require separate analytic arguments unless explicitly stated otherwise.

The first part of the paper is a formal enumeration.  A path of fixed height is decomposed into its minimal upper hull and a sequence of height-restricted inverted Dyck paths inserted below the hull.  Combining this decomposition with the classical bounded-height area polynomials gives an exact expansion
\begin{equation}
  G(x,y,p,q)=1+\sum_{h\geq1}G_h(x,y,p,q),
\end{equation}
where $x$ and $y$ mark semi-length and height, and $p$ and $q$ mark capacity and area.  Each height summand $G_h$ is an explicit rational expression in the finite height-restricted area-polynomial denominators recalled in Section~\ref{sec:height-restricted}.  Thus the finite-height components are standard.  The resulting expression couples area and capacity through the upper-hull substitution, gives symmetric denominator factors, and leads to the analytic problem of controlling the infinite family of height summands.

The second part is analytic.  Each insertion level produces a possible positive denominator singularity in the length variable $x$.  We call this possible singularity a branch; the active branch is the one that attains the smallest positive value and hence determines the candidate radius.  The issue is whether the infinite height sum can create an earlier singularity than these finite-level denominator branches.  For $0<p,q<1$, normal convergence rules out such earlier singularities.

More precisely, after writing $r=q^2/p^2$, the $k$th branch can be represented through the Perron eigenvalue of a geometrically weighted path matrix.  If $\Lambda_h(r)$ denotes the square of that Perron eigenvalue for the $h$-vertex path, then the radius in the open square is
\begin{equation}
  R(p,q)=\min_{k\geq1}
  \frac{1}{p^{2k-1}q\,\Lambda_{k+1}(q^2/p^2)}.
\end{equation}
The theorem gives the radius in the open square; local singular expansions and boundary coefficient asymptotics lie outside its scope.

The proof has two main ingredients.  First, we establish normal convergence of the height expansion on compact subsets lying below the first positive branch, so the infinite sum is analytic before the finite-level obstruction.  Second, we prove a log-submodularity statement for spectral radii of overlapping intervals in a positively weighted path.  In the geometric specialisation this yields the log-concavity inequality for the Perron-root sequence that orders the branch competition.  Consequently, for fixed $q^2/p^2$, active branches can change only at neighbouring crossings, with ties allowed when several branches meet at the same value.

The diagonal specialisation $p=q=s$ makes this branch selection explicit.  In that case the bounded-height denominator polynomials reduce to Chebyshev polynomials, as in the classical finite-height Catalan continued fractions.  For $0<s<1$ the open-square radius theorem gives a lower envelope of explicit branches, with adjacent crossings accumulating at $(s,x)=(1,1/4)$.  The resulting envelope satisfies
\begin{equation}
  x_c(s)=\frac14+\frac{3\pi^{2/3}}4(1-s)^{2/3}+o((1-s)^{2/3}),
  \qquad s\uparrow1.
\end{equation}
We refer to this as a branch-envelope accumulation law; it concerns the lower envelope of branch locations rather than coefficient-asymptotic transfer.

The height expansion also has a symmetry at the level of the chosen unreduced denominators.  The full generating function is asymmetric under $p\leftrightarrow q$, while the denominator factors of the height summands can be arranged into unreduced polynomials $\R_h(x;p,q)$ satisfying $\R_h(x;p,q)=\R_h(x;q,p)$.  This symmetry organises the branch loci at the denominator level rather than at the coefficient level.

The paper is organised as follows.  Section~\ref{sec:model} defines the statistics and the four-variable generating function.  Section~\ref{sec:height-restricted} recalls the height-restricted area-weighted Dyck path polynomials used in the solution.  Section~\ref{sec:exact-solution} gives the exact height-decomposed generating function and isolates the analytic issue caused by the infinite sum.  Section~\ref{sec:singularity-symmetry} proves the denominator symmetry.  Section~\ref{sec:diagonal} treats the equal-weight hull-area specialisation, including the Chebyshev branch formulae, branch crossings and accumulation at $(1,1/4)$.  Section~\ref{sec:two-weight} proves the open-square radius theorem, using the weighted-path reformulation, the Perron-root log-concavity result, neighbouring-branch selection, and the normal-convergence estimate.  Section~\ref{sec:discussion} concludes; the first-moment consistency check and two auxiliary derivations are recorded in the appendices.

\section{The model}
\label{sec:model}

A Dyck path is a directed walk on $\bbZ^2$ starting at $(0,0)$, ending on the line $y=0$, never visiting a vertex with negative $y$-coordinate, and using steps $(1,1)$ and $(1,-1)$ \cite{deutsch1999}.  If $\pi$ has $2n$ steps, we call $n=n(\pi)$ its semi-length.  Its height $h=h(\pi)$ is the maximum $y$-coordinate attained by the path.

We use the \emph{proper area} $a(\pi)$, defined as the sum of the starting heights of all steps of the path.  Equivalently, this is the number of triangular plaquettes below the path.  This differs by a simple affine change from the diamond-area convention used in some treatments of area-weighted Dyck paths: if $d(\pi)$ is the number of diamond plaquettes below the path, then
\begin{equation}
  a(\pi)=2d(\pi)+n(\pi).
\end{equation}

We now give a formal definition of the upper hull and water capacity.  Let $h=h(\pi)$.  For each $1\leq k\leq h$, mark the first up-step of $\pi$ from level $k-1$ to level $k$ and the last down-step of $\pi$ from level $k$ to level $k-1$.  These marked steps occur in the order
\[
  U_1,U_2,\ldots,U_h,D_h,\ldots,D_2,D_1.
\]
Between consecutive marked steps there are gaps at fixed ceiling level $k$: for $1\leq k<h$ there is one left gap after $U_k$ and before $U_{k+1}$ and one right gap after $D_{k+1}$ and before $D_k$, while for $k=h$ there is a single top gap after $U_h$ and before $D_h$.  Each such gap starts and ends at level $k$ and never rises above level $k$.

The \emph{upper hull} is the piecewise-linear water-filling envelope obtained by replacing every gap at level $k$ by the horizontal segment at height $k$, while keeping the marked up- and down-steps.  Reflecting a gap at level $k$ in the line $y=k$ and translating it to the horizontal axis gives an ordinary Dyck path of height at most $k$.  The \emph{water capacity} $w(\pi)$ is the sum, over all gaps, of the proper areas of these reflected Dyck paths.  Equivalently, $w(\pi)$ is the number of triangular plaquettes lying above $\pi$ and below the upper hull.  The hull area, denoted by $\ell(\pi)$, is
\begin{equation}
  \ell(\pi)=a(\pi)+w(\pi).
\end{equation}
When all gaps are empty, the upper hull is the minimal hull $U^hD^h$.  Figure~\ref{fig:model} illustrates the area, retained water, and hull area for a small path.

\begin{figure}[!htbp]
\centering
\includegraphics[width=.70\textwidth]{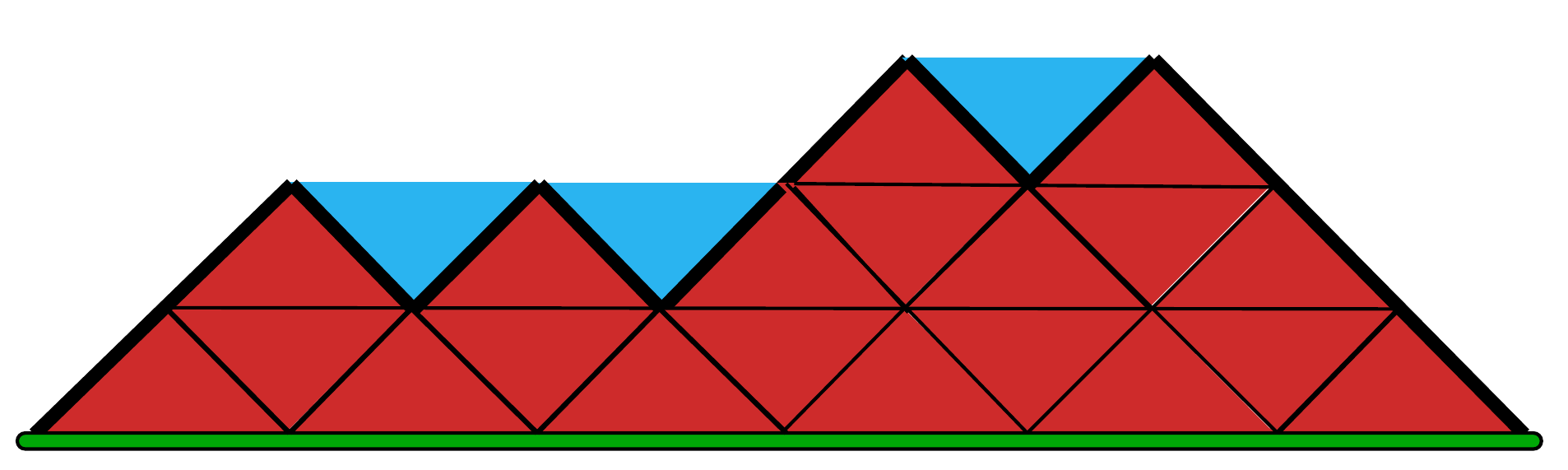}
\caption{A Dyck path of semi-length $n=6$, height $h=3$, area $a=20$ below the path, and water capacity $w=3$.  The area plaquettes indicate proper area, and the water plaquettes indicate retained water below the upper hull.  The hull area is $a+w=23$.}
\label{fig:model}
\end{figure}

Let $\Dyck$ denote the set of all Dyck paths.  For indeterminates $x,y,p,q$, marking semi-length, height, water capacity and area, respectively, define the generating function
\begin{equation}
  G(x,y,p,q)=\sum_{\pi\in\Dyck} x^{n(\pi)}y^{h(\pi)}p^{w(\pi)}q^{a(\pi)}.
\end{equation}
For $N\geq0$, we also use the height-truncated generating function
\begin{equation}
  G_N(x,y,p,q)=\sum_{\substack{\pi\in\Dyck\\ h(\pi)\leq N}} x^{n(\pi)}y^{h(\pi)}p^{w(\pi)}q^{a(\pi)}.
\end{equation}
The untruncated function is the formal limit $G=\lim_{N\to\infty}G_N$; unless analytic convergence is explicitly invoked, identities involving $G$ are understood coefficientwise as formal power series in $x$.

Two small examples fix the conventions.  The contribution of paths of height exactly one is
\begin{equation}
  G_1(x,y,p,q)=\sum_{n\geq1}x^nyp^{n-1}q^n
  =\frac{xyq}{1-pqx},
\end{equation}
because these paths are $(UD)^n$, with area $n$ and capacity $n-1$.  At semi-length two there are two paths:
\[
  \mathrm{UUDD}:\quad (a,w)=(4,0),
  \qquad
  \mathrm{UDUD}:\quad (a,w)=(2,1).
\]
Thus
\begin{equation}
  [x^2]G(x,1,p,q)=q^4+pq^2.
\end{equation}
The first few terms at $y=1$ are
\begin{align}
G(x,1,p,q)
&=1+qx+(q^4+pq^2)x^2\nonumber\\
&\quad +(q^9+pq^7+2pq^5+p^2q^3)x^3+O(x^4).
\end{align}
Here $O(x^4)$ denotes terms of degree at least $4$ in $x$, with coefficients in $p$ and $q$.

The specialisation $p=1$ gives area-weighted Dyck paths.  The specialisation $q=1$ gives capacity-weighted Dyck paths.  The diagonal specialisation $p=q=s$ gives
\begin{equation}
  G(x,y,s,s)=\sum_{\pi\in\Dyck} x^{n(\pi)}y^{h(\pi)}s^{\ell(\pi)},
\end{equation}
which is the generating function for Dyck paths weighted by hull area.

Table~\ref{tab:notation} collects the main notation used below.

\begin{table}[!htbp]
\centering
\small
\renewcommand{\arraystretch}{1.15}
\begin{tabular}{p{0.27\textwidth}p{0.63\textwidth}}
\hline
Symbol & Meaning \\
\hline
$\Dyck$ & set of Dyck paths \\
$n(\pi)$, $h(\pi)$ & semi-length and height of $\pi$ \\
$a(\pi)$, $w(\pi)$, $\ell(\pi)$ & proper area, water capacity, and hull area $a(\pi)+w(\pi)$ \\
$x,y,p,q$ & weights for semi-length, height, capacity, and area \\
$G$, $G_N$, $G_h$ & full, height-truncated, and exact-height generating functions \\
$D_h(X,Q)$ & bounded-height Dyck-path area generating function \\
$P_h(X,Q)$ & denominator polynomial for $D_h(X,Q)$ \\
$\R_h(x;p,q)$ & chosen unreduced symmetric denominator for height $h$ \\
$Q=p^2/q^2$, $r=q^2/p^2$ & reciprocal parameters used in the formal and spectral forms \\
$\Lambda_h(r)$ & square of the Perron eigenvalue of the weighted $h$-vertex path matrix \\
$x_k(p,q)$ & positive denominator branch indexed by insertion level $k$ \\
$R(p,q)$, $x_c(s)$ & open-square radius and its diagonal specialisation \\
\hline
\end{tabular}
\caption{Notation used in the model and in the later denominator and branch analysis.}
\label{tab:notation}
\end{table}

\section{Height-restricted Dyck paths}
\label{sec:height-restricted}

We use the following standard bounded-height Dyck-path construction.  For a non-negative integer $N$ and an integer $m$ with $0\leq m\leq N$, write $\qbinom{N}{m}{Q}$ for the Gaussian polynomial, characterised for generic $Q$ by
\begin{equation}
  \qbinom{N}{m}{Q}=\frac{(Q;Q)_N}{(Q;Q)_m(Q;Q)_{N-m}},
  \qquad (z;Q)_N=\prod_{j=0}^{N-1}(1-zQ^j),
\end{equation}
with $(z;Q)_0=1$ \cite[Appendix~I]{gasper2004,andrews1999}.  The quotient is understood polynomially in $Q$; in particular $\qbinom{N}{m}{1}=\binom{N}{m}$.  We set the $Q$-binomial coefficient to zero when $m<0$ or $m>N$.

For $h\geq0$ and auxiliary variables $X$ and $Q$, define the denominator polynomial
\begin{equation}
  P_h(X,Q)=\sum_{m=0}^{\lfloor h/2\rfloor}(-X)^m Q^{m(m-1)}\qbinom{h-m}{m}{Q}.
  \label{eq:P-def}
\end{equation}
Thus $P_0=P_1=1$, and for $h\geq2$ these polynomials satisfy
\begin{align}
  P_h(X,Q)&=P_{h-1}(QX,Q)-XP_{h-2}(Q^2X,Q),\label{eq:P-rec-1}\\
  P_h(X,Q)&=P_{h-1}(X,Q)-XQ^{h-2}P_{h-2}(X,Q).\label{eq:P-rec-2}
\end{align}
Both recurrences are the direct coefficientwise consequences of the two Gaussian Pascal identities \cite[Appendix~I]{gasper2004,andrews1999}.  This is the usual denominator sequence for the finite continued fractions enumerating height-restricted Dyck paths.

For $h\geq0$, let $D_h(X,Q)$ be the generating function for Dyck paths of height at most $h$, with $X$ marking semi-length and $Q$ marking diamond area.  Then
\begin{equation}
  D_h(X,Q)=\frac{P_h(QX,Q)}{P_{h+1}(X,Q)}.
  \label{eq:height-restricted-dyck}
\end{equation}
Since $P_{h+1}(0,Q)=1$, the quotient is interpreted as a formal power series in $X$; for fixed complex $Q$ it is also a rational function of $X$.  We use the form from \cite[Corollary~2]{owczarek2012}, equivalent to the usual finite continued fraction for height-restricted Dyck paths \cite{flajolet1980,goulden1986}.  Equivalently, the first-return decomposition gives $D_h(X,Q)=1/(1-XD_{h-1}(QX,Q))$, and \eqref{eq:height-restricted-dyck} follows from \eqref{eq:P-rec-1} by induction.

The first cases are
\begin{align*}
  &P_0(X,Q)=P_1(X,Q)=1,
  \qquad P_2(X,Q)=1-X,\\
  &P_3(X,Q)=1-(1+Q)X,
\end{align*}
and hence
\begin{equation*}
  D_0(X,Q)=1,
  \qquad
  D_1(X,Q)=\frac{1}{1-X},
  \qquad
  D_2(X,Q)=\frac{1-QX}{1-(1+Q)X}.
\end{equation*}
These examples are included only to fix the indexing convention for the denominator polynomials used below.

\section{Exact generating function}
\label{sec:exact-solution}

We now insert the height-restricted Dyck paths from \eqref{eq:height-restricted-dyck} into the upper-hull gaps defined in Section~\ref{sec:model}.  The following lemma is the combinatorial step; the product formula after it is the primary exact form of the generating function.

\begin{lemma}[Height-$h$ hull decomposition]
\label{lem:height-h-hull-decomposition}
A Dyck path of height exactly $h$ is uniquely obtained from the minimal hull $U^hD^h$ by inserting, for each $1\leq k<h$, two inverted Dyck paths of height at most $k$, and, for $k=h$, one inverted Dyck path of height at most $h$.  Empty insertions are allowed.  The minimal hull contributes $x^hy^hq^{h^2}$.  An inserted path at level $k$, with semi-length $m$ and diamond area $d$, contributes
\[
  x^m p^{m+2d} q^{(2k-1)m-2d},
\]
and hence has generating function
\[
  D_k\!\left(q^{2k-1}px,\frac{p^2}{q^2}\right).
\]
\end{lemma}

\begin{proof}
Mark the first up-step from level $k-1$ to level $k$ for each $1\leq k\leq h$, and the last down-step from level $k$ to level $k-1$ for each $1\leq k\leq h$.  These marked steps occur in the order $U_1,\ldots,U_h,D_h,\ldots,D_1$ and form the contracted skeleton $U^hD^h$.  The gaps between consecutive marked steps on the left and right sides at level $k<h$, and the single top gap at level $h$, are paths starting and ending at level $k$ and never rising above level $k$.  Reflecting such a gap in the line $y=k$ gives an ordinary Dyck path of height at most $k$, and the construction is reversible by inserting arbitrary such reflected paths.

The skeleton has semi-length $h$, height $h$, zero water capacity, and proper area $h^2$, giving the weight $x^hy^hq^{h^2}$.  Now consider one reflected gap path at level $k$.  If the reflected ordinary Dyck path has semi-length $m$ and diamond area $d$, then its proper area is $m+2d$.  This proper area is exactly the water capacity contributed by the original gap.  The horizontal strip below the level-$k$ hull over a base of length $2m$ contains $2km$ triangular plaquettes, that is, it has proper area $2km$.  Of these triangular plaquettes, $m+2d$ are occupied by water.  The remaining proper area below the original gap is therefore
\[
  2km-(m+2d)=(2k-1)m-2d.
\]
Thus the insertion contributes
\[
  x^m p^{m+2d}q^{(2k-1)m-2d}
  =(q^{2k-1}px)^m\left(\frac{p^2}{q^2}\right)^d.
\]
The factor $q^{2k-1}p$ is the weight per inserted semi-length unit, while the factor $p^2/q^2$ transfers one diamond-area unit from ordinary proper area to water capacity.  Summing over all reflected gaps of height at most $k$ gives $D_k(q^{2k-1}px,p^2/q^2)$.
\end{proof}

Let $G_h(x,y,p,q)$ be the contribution from paths of height exactly $h$.  The lemma gives, for $h\geq1$,
\begin{equation}
  G_h(x,y,p,q)=x^hy^hq^{h^2}
  \left[\prod_{k=1}^{h-1}D_k\!\left(q^{2k-1}px,p^2/q^2\right)\right]^2
  D_h\!\left(q^{2h-1}px,p^2/q^2\right),
  \label{eq:Gh-D}
\end{equation}
with the empty product interpreted as one.  Therefore the exact generating function is
\begin{equation}
  G(x,y,p,q)=1+\sum_{h\geq 1}G_h(x,y,p,q).
  \label{eq:G-sum-Gh}
\end{equation}
This is the main formal expression: each height summand is a product of standard bounded-height Dyck-path generating functions, with the substitution
\begin{equation}
  X=q^{2k-1}px,
  \qquad
  Q=\frac{p^2}{q^2}
  \label{eq:level-substitution}
\end{equation}
at insertion level $k$.

The coordinate boundaries are immediate from the definition and the hull decomposition.  Since every non-empty Dyck path has positive proper area,
\begin{equation}
  G(x,y,p,0)=1.
\end{equation}
Setting $p=0$ suppresses every positive-capacity insertion, leaving only the minimal hull $U^hD^h$ at height $h$, so
\begin{equation}
  G(x,y,0,q)=1+\sum_{h\geq1}(xy)^h q^{h^2}.
\end{equation}

Equivalently, substituting \eqref{eq:height-restricted-dyck} into \eqref{eq:Gh-D} gives an expression in the denominator polynomials $P_h$.

\begin{proposition}[Exact solution, expanded form]
Let $P_h$ be defined by \eqref{eq:P-def}.  For $q\ne0$, in the localized formal power-series ring with $q$ invertible, the height-truncated generating function is
\begin{equation}
  G_N(x,y,p,q)=
  \sum_{h=0}^{N}
  \frac{x^hy^hq^{h^2}
  \displaystyle\prod_{k=1}^{h}
  \frac{P_k(q^{2k-3}p^3x,p^2/q^2)^2}
       {P_k(q^{2k-3}px,p^2/q^2)^2}}
  {P_h(q^{2h-3}p^3x,p^2/q^2)
   P_{h+1}(q^{2h-1}px,p^2/q^2)}.
  \label{eq:exact-solution}
\end{equation}
The degenerate boundary $q=0$ is not represented by the displayed ratios and has $G_N(x,y,p,0)=1$.  At $p=0$, $G_N(x,y,0,q)=1+\sum_{h=1}^{N}(xy)^h q^{h^2}$.  Products over empty index sets in \eqref{eq:exact-solution} are interpreted as $1$.  For the $h=0$ term, the arguments of $P_0$ and $P_1$ are immaterial since $P_0=P_1=1$.  The unbounded-height generating function is the formal limit $N\to\infty$.
\end{proposition}

\begin{proof}
By Lemma~\ref{lem:height-h-hull-decomposition}, $G_h$ is given by \eqref{eq:Gh-D}.  Using \eqref{eq:height-restricted-dyck}, put
\begin{equation}
  A_k=P_k(q^{2k-3}p^3x,p^2/q^2),
  \qquad
  B_k=P_k(q^{2k-3}px,p^2/q^2).
\end{equation}
Then
\begin{equation}
  D_k\!\left(q^{2k-1}px,p^2/q^2\right)=\frac{A_k}{B_{k+1}},
\end{equation}
with $B_1=1$.  Multiplying over levels $1\leq k<h$ twice and over $k=h$ once gives \eqref{eq:exact-solution}, after relabelling the denominator product.  The term $h=0$ is the empty path.
\end{proof}

The identities above are first formal power-series identities.  For each coefficient of $x^n$, only finitely many heights contribute, so the infinite height sum is algebraically well defined.  For fixed non-zero complex $p$ and $q$, the same expression may also be viewed as an infinite sum of meromorphic functions of $x$.  Singularities can then come either from zeros of denominator factors in individual height summands or from failure of convergence of the infinite sum.  In the open square $0<p,q<1$, Proposition~\ref{prop:normal-convergence} below gives normal convergence up to the first positive termwise denominator singularity, eliminating this second source in that range.

\section{Symmetric denominator factors}
\label{sec:singularity-symmetry}

The full generating function is asymmetric in $p$ and $q$, as is clear from the definitions of area and water capacity.  Nevertheless, the rational summands in the height expansion admit a symmetric arrangement of the denominators.  We fix the following convention for the related notions used below.

Throughout this section, $\R_h$ denotes a chosen \emph{unreduced denominator representation}: it is the product of denominator factors obtained from the representation of the height-$h$ summand before cancellations with the numerator are made.  Its zero set is the corresponding denominator locus.  An actual pole of a height summand is a zero of this denominator representation which survives cancellation with the numerator.  For positive parameters, a positive denominator branch is a positive zero in the length variable $x$ arising from one of these factors, and an active branch is one which attains the first positive denominator value relevant to the radius theorem proved later.  The result in this section concerns this unreduced denominator representation, rather than the full generating function or its coefficients.

The required algebraic input is the inversion identity, valid for $Q\ne0$ and $0\leq m\leq n$,
\begin{equation}
  \qbinom{n}{m}{Q^{-1}}=Q^{-m(n-m)}\qbinom{n}{m}{Q},
\end{equation}
which implies, for $h\geq0$ and $Q\ne0$,
\begin{equation}
  P_h(X,Q)=P_h(Q^{h-2}X,Q^{-1}).
  \label{eq:P-inversion}
\end{equation}
The cases $h=0,1$ are trivial.  Thus the denominator identity below is first an identity for $pq\ne0$.

For $pq\ne0$ and $h\geq0$, define
\begin{equation}
  \R_h(x;p,q)=
  \left[\prod_{k=0}^{h}P_k(q^{2k-3}px,p^2/q^2)^2\right]
  P_{h+1}(q^{2h-1}px,p^2/q^2),
  \label{eq:R-def}
\end{equation}
and set $\R_{-1}(x;p,q)=1$.  The factors with $k=0$ or $k=1$ are constant whenever the displayed negative powers of $q$ occur.  Applying \eqref{eq:P-inversion} factor by factor gives
\begin{equation}
  \R_h(x;p,q)=\R_h(x;q,p).
  \label{eq:R-symmetry}
\end{equation}
Moreover, \eqref{eq:exact-solution} can be rewritten as
\begin{equation}
  G_N(x,y,p,q)=\sum_{h=0}^{N}
  x^hy^hq^{h^2}\frac{\R_{h-1}(p^2x;p,q)}{\R_h(x;p,q)},
  \label{eq:G-R-form}
\end{equation}
where $\R_{-1}$ is used only in the $h=0$ term.

For the first few values, these denominators are
\begin{equation}
  \R_0(x;p,q)=1,
  \qquad
  \R_1(x;p,q)=1-pqx,
\end{equation}
and
\begin{equation}
  \R_2(x;p,q)=(1-pqx)^2\bigl(1-pq(p^2+q^2)x\bigr).
\end{equation}
The symmetry in $p$ and $q$ is visible already in these examples, but the examples should also be read with the convention above: zeros of $\R_h$ are candidate denominator zeros rather than guaranteed uncancelled poles of the corresponding summand.

It follows from \eqref{eq:R-symmetry} that the termwise denominator locus in the $x$-plane is symmetric under $p\leftrightarrow q$.  In particular, if, for positive $p$ and $q$ in a regime where the minimum exists, $x_{\mathrm{term}}(p,q)$ denotes the smallest positive $x$-value at which the unreduced denominator representation of some height summand vanishes, then
\begin{equation}
  x_{\mathrm{term}}(p,q)=x_{\mathrm{term}}(q,p).
  \label{eq:xc-symmetry}
\end{equation}
Equivalently, outside such a regime one may use the corresponding infimum.  The open-square radius theorem below identifies when the active positive branch from this denominator family controls the full height sum.

\begin{remark}[Boundary values]
Although \eqref{eq:R-def} is written for $pq\ne0$, the factors extend polynomially in $x,p,q$.  On the coordinate axes this extension gives $\R_h(x;p,q)=1$.  This boundary convention is algebraic only; the analytic radius theorem in Section~\ref{sec:two-weight} is stated for $0<p,q<1$.
\end{remark}

\section{The diagonal hull-area specialisation}
\label{sec:diagonal}

We now set $p=q=s$.  This specialisation is explicit enough to display the branch selection that will be treated in general in Section~\ref{sec:two-weight}.  The finite-height ingredient itself is classical: when the second argument of $P_h$ in \eqref{eq:exact-solution} is $1$, the bounded-height denominator polynomials reduce to Chebyshev polynomials.  Let $U_h$ denote the Chebyshev polynomial of the second kind.  Then
\begin{equation}
  P_h(X,1)=\sum_{m=0}^{\lfloor h/2\rfloor}(-X)^m\binom{h-m}{m}
  =X^{h/2}U_h\!\left(\frac{1}{2\sqrt X}\right),
  \label{eq:P-cheb}
\end{equation}
Indeed, both sides of \eqref{eq:P-cheb} have initial values $1,1$ for $h=0,1$ and satisfy the same recurrence
\[
  F_h(X)=F_{h-1}(X)-XF_{h-2}(X),
\]
using \eqref{eq:P-rec-2} at $Q=1$ and the three-term recurrence for $U_h$.  This is an identity in $X$ after cancellation of half-powers; for $X>0$ the positive square root is used.  Equation~\eqref{eq:exact-solution} becomes the following explicit Chebyshev form.

\begin{proposition}[Diagonal hull-area generating function]
For $p=q=s$ with $s\ne0$,
\begin{equation}
G(x,y,s,s)=\frac{1}{\sqrt x}\sum_{h\geq0}
\frac{y^h
\displaystyle\prod_{k=0}^{h}
\left[
\frac{U_k\!\left(s^{-k}/(2\sqrt x)\right)}
     {U_k\!\left(s^{1-k}/(2\sqrt x)\right)}
\right]^2}
{U_h\!\left(s^{-h}/(2\sqrt x)\right)
 U_{h+1}\!\left(s^{-h}/(2\sqrt x)\right)}.
\label{eq:cheb-gf}
\end{equation}
The square-root expression is shorthand for the corresponding polynomial identity obtained from \eqref{eq:P-cheb}; the apparent singularity at $x=0$ cancels termwise.  For analytic statements near positive $x$, the positive branch of $\sqrt{x}$ is used.  At $s=0$, the boundary value is $G(x,y,0,0)=1$.
\end{proposition}

The formula \eqref{eq:cheb-gf} is intentionally left in the same unreduced denominator convention as in Section~\ref{sec:singularity-symmetry}.  Some displayed denominator zeros are cancelled by numerator factors, and repeated denominator factors may represent the same locus.  The branches below are the positive denominator zeros that remain relevant to the lower envelope after this convention is taken into account.

The zeros of $U_h$ are explicit.  Since
\begin{equation}
  U_h(\cos\theta)=\frac{\sin((h+1)\theta)}{\sin\theta},
\end{equation}
the zeros are obtained at $\theta=j\pi/(h+1)$, $1\leq j\leq h$.  Writing $k\geq1$ for the branch coming from the denominator factor indexed by $k+1$, the closest positive denominator zero gives
\begin{equation}
  x_k(s)=\frac{1}{4s^{2k}\cos^2\left(\frac{\pi}{k+2}\right)},
  \qquad k\geq1.
  \label{eq:xk}
\end{equation}
The first entries are
\begin{center}
\begin{tabular}{c|c}
quantity & value \\
\hline
$x_1(s)$ & $s^{-2}$ \\
$x_2(s)$ & $1/(2s^4)$ \\
$s_1$ & $1/\sqrt2$
\end{tabular}
\end{center}
where $s_1$ is the first adjacent crossing value defined below.  These small cases check the indexing of the branch formula.

The following conclusion about the radius for $0<s<1$ uses the normal-convergence result proved later as Theorem~\ref{thm:open-square-radius}.  Consequently, after applying that theorem, the length radius in the diagonal model is
\begin{equation}
  x_c(s)=\min_{k\geq1} x_k(s),
  \qquad 0<s<1.
  \label{eq:xc-diagonal}
\end{equation}
At $s=1$ this minimum is replaced by the limiting value $x_c(1)=1/4=\inf_{k\geq1}x_k(1)$, in agreement with the Catalan generating function $G(x,1,1,1)$.
For $s>1$, the coefficient of $x^n$ in $G(x,1,s,s)$ is at least the contribution $s^{n^2}$ of the path $U^nD^n$, so the length generating function has zero radius of convergence.  This gives the zero-radius regime.

Consecutive branches $x_k(s)$ and $x_{k+1}(s)$ cross at
\begin{equation}
  s_k=\frac{\cos(\pi/(k+2))}{\cos(\pi/(k+3))},
  \label{eq:sk}
\end{equation}
with corresponding $x$-coordinate $\xi_k=x_k(s_k)$,
\begin{equation}
  \xi_k=\frac{1}{4}
  \left(\frac{\cos(\pi/(k+2))}{\cos(\pi/(k+3))}\right)^{-2k}
  \cos^{-2}\!\left(\frac{\pi}{k+2}\right).
  \label{eq:Xk}
\end{equation}
Therefore the radius is a piecewise analytic lower envelope of branches associated with simple positive denominator zeros.  If one writes the associated grand-canonical free energy as
\begin{equation}
  \kappa(s)=-\log x_c(s),
\end{equation}
then $\kappa$ is continuous at adjacent crossings, while its one-sided derivatives differ because
\begin{equation}
  \frac{\dd}{\dd s}\log x_k(s)=-\frac{2k}{s}.
\end{equation}
Thus the diagonal specialisation has an infinite sequence of first-derivative discontinuities in this free-energy function.  The statement concerns the branch envelope of the radius and leaves coefficient asymptotics across the crossings outside its scope.

The transition points accumulate at $s=1$ and $x=1/4$.  All $O$-terms in the next two displays are as $k\to\infty$, and the final $o$-term is as $s\uparrow1$.  As $k\to\infty$,
\begin{align}
  s_k&=1-\frac{\pi^2}{k^3}+O(k^{-4}),\label{eq:sk-asymp}\\
  \xi_k&=\frac14+\frac{3\pi^2}{4k^2}+O(k^{-3}).\label{eq:Xk-asymp}
\end{align}
To pass from the crossing points to the full lower envelope, put $t=-\log s$.  Then, uniformly for large $k$,
\begin{equation}
  \log(4x_k(s))=2kt-2\log\cos\!\left(\frac{\pi}{k+2}\right)
  =2kt+\frac{\pi^2}{(k+2)^2}+O(k^{-4}).
\end{equation}
The minimising index tends to infinity as $t\downarrow0$, since for every fixed $k$ the second term stays positive, while choosing $k\asymp t^{-1/3}$ makes the right-hand side tend to $0$.  Hence the discrete minimum has the same leading order as the continuous minimum
\begin{equation}
  \min_{u>0}\left(2tu+\frac{\pi^2}{u^2}\right)=3\pi^{2/3}t^{2/3},
\end{equation}
attained at $u=(\pi^2/t)^{1/3}$.  The shift from $k$ to $k+2$ and the $O(k^{-4})$ remainder are $o(t^{2/3})$ at this scale.  Therefore
\begin{equation}
  \min_{k\geq1}\log(4x_k(s))=3\pi^{2/3}t^{2/3}+o(t^{2/3}).
\end{equation}
Since $t\sim1-s$ and $\exp z=1+z+o(z)$, this gives the branch-envelope asymptotic
\begin{equation}
  x_c(s)=\frac14+\frac{3\pi^{2/3}}{4}(1-s)^{2/3}+o\bigl((1-s)^{2/3}\bigr),
  \qquad s\uparrow1.
  \label{eq:xc-diagonal-asymp}
\end{equation}
The exponent $2/3$ agrees with the exponent appearing in the ordinary area-weighted Dyck path transition \cite{haug2015,haug2017}; the statement proved here is the branch-envelope accumulation law above.  Figure~\ref{fig:diagonal-singularities} shows the first branch curves and the accumulation near $(s,x)=(1,1/4)$.

\begin{figure}[t]
\centering
\begin{minipage}{.48\textwidth}
\centering
\includegraphics[width=.95\textwidth]{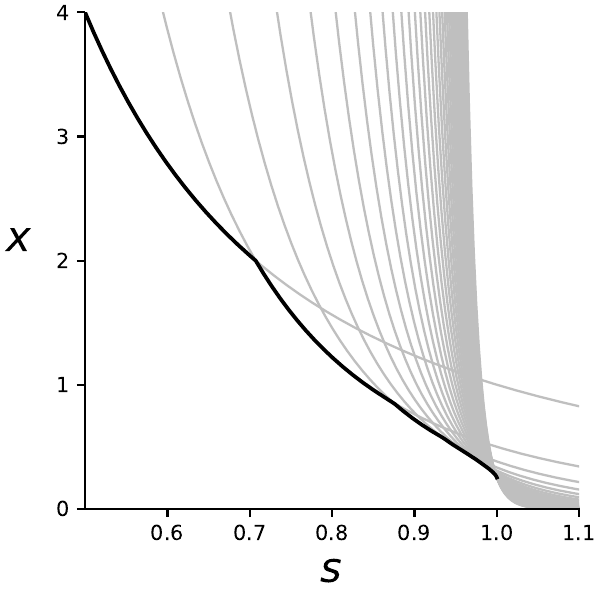}
\end{minipage}\hfill
\begin{minipage}{.48\textwidth}
\centering
\includegraphics[width=.95\textwidth]{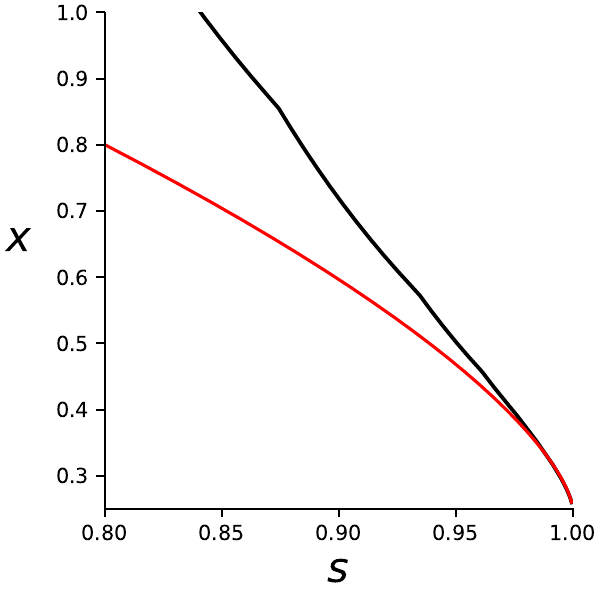}
\end{minipage}
\caption{Positive denominator branches for the diagonal hull-area model.  The horizontal axis in both panels is the diagonal weight $s$.  The radius is the lower envelope of the displayed branches for $0<s<1$.  The right panel zooms near $(s,x)=(1,1/4)$ and compares with the branch-envelope asymptotic \eqref{eq:xc-diagonal-asymp}.}
\label{fig:diagonal-singularities}
\end{figure}

\begin{remark}[First moment at the unweighted point]
A short consistency check at the unweighted point is recorded in \ref{app:first-moment}.  This check is independent of the proof of the open-square radius theorem and confirms that differentiating the diagonal weight recovers the hull-area first moment as the sum of the known proper-area and water-capacity first moments.
\end{remark}

\section{Two-parameter branches and the radius theorem}
\label{sec:two-weight}

We now return to general $p$ and $q$.  The exact expression \eqref{eq:exact-solution} gives an infinite sum of rational functions built from the polynomials $P_h$.  We analyse the denominator branches coming from the factors of the summands and determine the branch selected in the open square.  This section shows that, for $0<p,q<1$ and $y=1$, the infinite height sum is analytic before the first positive denominator branch: the radius of the full generating function is exactly the smallest positive branch value.

Using the symmetric form of the exact solution, and with $pq\ne0$ algebraically or $p,q>0$ for the positive-real branch interpretation, the relevant denominator factor for the $k$th branch can be written as
\begin{equation}
  \Delta_k(x;p,q)=P_{k+1}\!\left(p^{2k-1}qx,\frac{q^2}{p^2}\right),
  \qquad k\geq1.
  \label{eq:Delta-k}
\end{equation}
The corresponding termwise denominator branch surfaces are the algebraic surfaces
\begin{equation}
  \Delta_k(x;p,q)=0.
\end{equation}
By \eqref{eq:R-symmetry}, the family of these surfaces is invariant under $p\leftrightarrow q$.  This identity is the algebraic source of the symmetry of the two-parameter denominator surfaces.  For positive $p$ and $q$, let $x_k(p,q)$ denote the smallest positive real solution of $\Delta_k(x;p,q)=0$.

The first two positive branches are
\begin{equation}
  x_1(p,q)=\frac1{pq},
  \qquad
  x_2(p,q)=\frac1{pq(p^2+q^2)}.
  \label{eq:x1-x2}
\end{equation}
Their relevant positive open-square intersection is the curve
\begin{equation}
  p^2+q^2=1.
  \label{eq:first-general-transition}
\end{equation}
Eliminating $x$ between the second and third denominator factors gives the next adjacent transition curve
\begin{equation}
0=p^4q^4-p^6-2p^4q^2-2p^2q^4-q^6+p^4+2p^2q^2+q^4.
\label{eq:second-general-transition}
\end{equation}
The derivation and a convenient parametrisation of its open-square component are given in \ref{app:second-transition}.  Further transition curves are obtained in the same way, although their algebraic expressions grow rapidly in size.

The recurrence \eqref{eq:P-rec-2} gives a structural interpretation of the termwise branches.  Transfer-matrix and determinant formulations for bounded-height path enumeration are classical; the weighted-path form below is the version needed for the present singularity analysis \cite{bacher2013,ouvry2021,cigler2024}.  For $r>0$, take $A_0(r)$ to be the empty matrix and $A_1(r)$ to be the $1\times1$ zero matrix.  For $h\geq2$, let $A_h(r)$ be the $h\times h$ symmetric tridiagonal matrix with zero diagonal and off-diagonal entries
\begin{equation}
  1,\, r^{1/2},\, r,\, r^{3/2},\,\ldots,\, r^{(h-2)/2},
\end{equation}
where $r^{1/2}$ denotes the positive real square root.
Equivalently,
\begin{equation}
  (A_h(r))_{i,i+1}=(A_h(r))_{i+1,i}=r^{(i-1)/2},
  \qquad 1\leq i\leq h-1,
\end{equation}
and all other entries are zero.  For $h\geq1$, write $\lambda_h(r)$ for the spectral radius of $A_h(r)$; thus $\lambda_1(r)=0$, and for $h\geq2$ this is the Perron eigenvalue in the standard Perron--Frobenius sense \cite[Chapter~8]{hornjohnson2013}.  Put
\begin{equation}
  \Lambda_h(r)=\lambda_h(r)^2.
\end{equation}

\begin{proposition}[Weighted-path representation]
\label{prop:weighted-path}
For $h\geq0$ and $r>0$,
\begin{equation}
  P_h(X,r)=\det(I-\sqrt X\,A_h(r)).
  \label{eq:P-weighted-path-det}
\end{equation}
Here $I$ denotes the identity matrix of the appropriate size, and the determinant of a $0\times0$ matrix is interpreted as one.  The right-hand side is understood as the resulting polynomial in $X$, independent of the choice of square-root branch.  For $h\geq2$, the zeros of $P_h(X,r)$ are positive real numbers, and the smallest positive zero is $\Lambda_h(r)^{-1}$.
\end{proposition}

\begin{proof}
Let $F_h(X,r)=\det(I-\sqrt X\,A_h(r))$.  Expansion of the tridiagonal determinant gives the recurrence
\begin{equation}
  F_h(X,r)=F_{h-1}(X,r)-Xr^{h-2}F_{h-2}(X,r),
\end{equation}
with initial values $F_0=F_1=1$.  This agrees with \eqref{eq:P-rec-2}.  The matrix $A_h(r)$ is real symmetric and bipartite, so its non-zero eigenvalues occur in pairs $\pm\lambda$.  Hence the determinant in \eqref{eq:P-weighted-path-det} is a polynomial in $X$, whose zeros are $X=\lambda^{-2}$ for the non-zero eigenvalues.  The closest zero is obtained from the largest positive eigenvalue, which is the Perron eigenvalue since $A_h(r)$ is an irreducible non-negative weighted path matrix for $h\geq2$.  Since $A_h(r)$ is a Jacobi matrix with strictly positive off-diagonal entries, its eigenvalues are simple \cite[Chapters~4 and~8]{hornjohnson2013}.  Hence the positive $X$-zeros of $P_h(X,r)$, in particular $X=\Lambda_h(r)^{-1}$, are simple.
\end{proof}

For $p,q>0$, put
\begin{equation}
  r=\frac{q^2}{p^2}.
\end{equation}
The positive branch value associated with the $k$th branch is
\begin{equation}
  x_k(p,q)=\frac{1}{p^{2k-1}q\,\Lambda_{k+1}(r)}
  =\frac{1}{p^{2k}\sqrt r\,\Lambda_{k+1}(r)}.
  \label{eq:xk-weighted-path}
\end{equation}
The denominator appearing directly in the corresponding $D_k$-factor is
\begin{equation}
  P_{k+1}(q^{2k-1}px,p^2/q^2).
\end{equation}
For $pq\ne0$, the inversion identity \eqref{eq:P-inversion} shows that this polynomial has the same zeros in $x$ as \eqref{eq:Delta-k}.  Equivalently,
\begin{equation}
  x_k(p,q)=\frac{1}{q^{2k-1}p\,\Lambda_{k+1}(p^2/q^2)}.
  \label{eq:xk-weighted-path-alt}
\end{equation}
The Chebyshev formula is recovered on the diagonal.  For $r=1$ the matrix $A_h(1)$ is the adjacency matrix of the ordinary path graph, so
\begin{equation}
  \lambda_h(1)=2\cos\left(\frac{\pi}{h+1}\right),
  \qquad
  \Lambda_h(1)=4\cos^2\left(\frac{\pi}{h+1}\right),
\end{equation}
and \eqref{eq:xk-weighted-path} becomes \eqref{eq:xk}.

For fixed $r>0$, write $\sigma=p^2$.  Since the factor $\sqrt r$ in \eqref{eq:xk-weighted-path} is common to all branches, selecting the smallest termwise branch value is equivalent to maximising
\begin{equation}
  \sigma^k\Lambda_{k+1}(r),\qquad k\geq1.
  \label{eq:weighted-path-envelope}
\end{equation}
The crossing of adjacent branches $k$ and $k+1$ occurs at
\begin{equation}
  \sigma=\frac{\Lambda_{k+1}(r)}{\Lambda_{k+2}(r)}.
  \label{eq:weighted-path-crossing}
\end{equation}
Thus the off-diagonal problem is a deformation of the Chebyshev path-spectrum problem.  The unweighted path has been replaced by a geometrically weighted path.  The Perron-root log-concavity needed to order the branches follows from an interval inequality for weighted paths.

We use the following interval inequality.  It is close in spirit to broader spectral submodularity and spectral-radius convexity results for principal submatrices and non-negative matrices \cite{friedland1975,friedland2013}, but here we need the elementary path-specific form below.  Let a path on vertices $1,\ldots,n+1$ have positive edge weights $a_1,\ldots,a_n$.  For an interval of vertices $[i,j]$, let $\rho_{i,j}$ denote the spectral radius of the weighted adjacency matrix induced by this interval, with $\rho_{i,i}=0$.

\begin{lemma}[Interval log-submodularity for weighted paths]
\label{lem:interval-log-submodularity}
For $n\geq2$,
\begin{equation}
  \rho_{1,n}\rho_{2,n+1}
  \geq
  \rho_{2,n}\rho_{1,n+1}.
  \label{eq:interval-log-submodularity}
\end{equation}
\end{lemma}

\begin{proof}[Proof sketch]
The proof compares Perron eigenvectors on the three overlapping intervals $[1,n]$, $[2,n]$, and $[2,n+1]$.  A discrete Wronskian comparison shows that one quotient of eigenvectors is monotone decreasing across the overlap, while the other is monotone increasing.  These two monotonicities allow one to construct a positive test vector on $[1,n+1]$ whose Collatz--Wielandt ratios are bounded by $\rho_{1,n}\rho_{2,n+1}/\rho_{2,n}$.  The Collatz--Wielandt formula then gives \eqref{eq:interval-log-submodularity}.  The full verification is given in \ref{app:interval-log-submodularity}.
\end{proof}

The preceding lemma uses the path order and the overlap of the intervals.  Applying it to the geometrically weighted path gives the required log-concavity.

\begin{theorem}[Perron-root log-concavity]
\label{thm:lambda-log-concavity}
For every $h\geq2$ and every $r>0$,
\begin{equation}
  \lambda_h(r)^2\geq \lambda_{h-1}(r)\lambda_{h+1}(r),
  \label{eq:lambda-log-concavity}
\end{equation}
and hence
\begin{equation}
  \Lambda_h(r)^2\geq \Lambda_{h-1}(r)\Lambda_{h+1}(r).
  \label{eq:Lambda-log-concavity}
\end{equation}
\end{theorem}

\begin{proof}
For $h=2$ the inequality is trivial, since $\lambda_1(r)=0$.  For $h\geq3$, apply Lemma~\ref{lem:interval-log-submodularity} to the path on vertices $1,\ldots,h+1$ with edge weights $a_i=r^{(i-1)/2}$.  The interval $[2,h+1]$ is $r^{1/2}$ times the original $h$-vertex geometric path, and $[2,h]$ is $r^{1/2}$ times the original $(h-1)$-vertex geometric path.  Thus
\begin{align}
  \rho_{1,h}&=\lambda_h(r),
  &\rho_{2,h+1}&=r^{1/2}\lambda_h(r),\nonumber\\
  \rho_{2,h}&=r^{1/2}\lambda_{h-1}(r),
  &\rho_{1,h+1}&=\lambda_{h+1}(r).
\end{align}
Substitution into \eqref{eq:interval-log-submodularity} gives \eqref{eq:lambda-log-concavity}.  Squaring gives \eqref{eq:Lambda-log-concavity}.
\end{proof}

\begin{corollary}[Neighbouring branch selection]
\label{cor:neighbouring-branches}
Fix $r>0$ and $0<\sigma<\min(1,r^{-1})$.  The sequence
\begin{equation}
  B_k(\sigma,r)=\sigma^k\Lambda_{k+1}(r),
  \qquad k\geq1,
\end{equation}
tends to zero as $k\to\infty$ and is unimodal in $k$.  Its maximum is attained.  The set of maximisers is an interval of integers, possibly a single point, and it can change only at values of $\sigma$ for which adjacent terms are equal.  The adjacent crossing values
\begin{equation}
  \sigma_k(r)=\frac{\Lambda_{k+1}(r)}{\Lambda_{k+2}(r)}
  \label{eq:sk-r}
\end{equation}
are nondecreasing in $k$.  The statement allows degeneracies: several adjacent equalities may occur at the same value of $\sigma$.
\end{corollary}

\begin{proof}
Put $L_k=\Lambda_{k+1}(r)$.  Theorem~\ref{thm:lambda-log-concavity} says that $L_k$ is log-concave, so $L_{k+1}/L_k$ is nonincreasing.  Hence
\begin{equation}
  \frac{B_{k+1}(\sigma,r)}{B_k(\sigma,r)}
  =\sigma\frac{L_{k+1}}{L_k}
\end{equation}
is nonincreasing in $k$.  The sequence $B_k(\sigma,r)$ is therefore unimodal in the weak sense: it can increase, then remain flat for a finite interval, and then decrease.  Consequently the maximising indices form an interval.  A change of this interval can occur only when one of the adjacent ratios is equal to $1$, equivalently when $B_{k+1}=B_k$, which gives \eqref{eq:sk-r}.  The argument uses only weak monotonicity, so simultaneous adjacent equalities are allowed.  The monotonicity of $\sigma_k(r)$ is the reciprocal form of the monotonicity of $L_{k+1}/L_k$.  Since $\sigma<\min(1,r^{-1})$, both $\sigma<1$ and $\sigma r<1$; the elementary bound $\Lambda_{k+1}(r)\leq4\max(1,r^{k-1})$ gives $B_k(\sigma,r)\to0$, so the maximum is attained.
\end{proof}

We next exclude singularities caused by loss of convergence of the infinite height sum.  Define
\begin{equation}
  x_*(p,q)=\min_{k\geq1}x_k(p,q),
  \qquad 0<p,q<1,
  \label{eq:xstar-pq}
\end{equation}
where $x_k(p,q)$ is given by \eqref{eq:xk-weighted-path}.  This minimum is attained.  Indeed, if $b<1$ and $p,q\leq b$, then
\begin{equation}
  p^{2k-1}q\,\Lambda_{k+1}(q^2/p^2)\leq 4b^{2k},
  \label{eq:xk-uniform-divergence}
\end{equation}
using the bound $\Lambda_{k+1}(r)\leq4\max(1,r^{k-1})$.  Hence $x_k(p,q)$ tends to infinity uniformly on compact subsets of the open square.  On such a compact set, $x_*(p,q)$ is therefore the minimum of finitely many continuous functions $x_k$, and is continuous and attained.

\begin{proposition}[Normal convergence below the first positive branch]
\label{prop:normal-convergence}
Let $K$ be a compact subset of $(0,1)^2$, and let $0\leq Y<\infty$.  For every $\epsilon>0$ with $\epsilon<\inf_{(p,q)\in K}x_*(p,q)$, there exist constants $C<\infty$ and $0<\theta<1$ such that
\begin{equation}
  |G_h(x,y,p,q)|\leq C\theta^h
  \label{eq:Gh-tail-bound}
\end{equation}
uniformly for $h\geq1$, $(p,q)\in K$, $|y|\leq Y$, and $|x|\leq x_*(p,q)-\epsilon$.
\end{proposition}

\begin{proof}
Choose $0<b_-<b_+<1$ such that $b_-\leq p,q\leq b_+$ on $K$.  Since $x_*(p,q)\leq x_1(p,q)=1/(pq)$, the region in question satisfies $|x|\leq M:=b_-^{-2}$.  The compactness assumption is used only through these two-sided bounds on $p,q$ and through the uniform separation from the finitely many small denominator zeros below.

Write
\begin{equation}
  \widehat X_k=q^{2k-1}px,
  \qquad
  Q=\frac{p^2}{q^2}.
\end{equation}
The denominator of $D_k(\widehat X_k,Q)$ is
\begin{equation}
  P_{k+1}(\widehat X_k,Q)=\det(I-\sqrt{\widehat X_k}A_{k+1}(Q)).
\end{equation}
The determinant is a polynomial in $x$, so the choice of the square-root branch is immaterial.  For estimating it on a complex disk we choose any square root and bound the moduli of the matrix entries.  The edge moduli of the matrix $\sqrt{\widehat X_k}A_{k+1}(Q)$ are
\begin{equation}
  |x|^{1/2}p^{i-1/2}q^{k-i+1/2},
  \qquad 1\leq i\leq k,
\end{equation}
and hence are at most $M^{1/2}b_+^k$.  The numerator determinant $P_k(Q\widehat X_k,Q)$ has edge moduli
\begin{equation}
  |x|^{1/2}p^{i+1/2}q^{k-i-1/2},
  \qquad 1\leq i\leq k-1,
\end{equation}
which obey the same bound.  Since a symmetric tridiagonal matrix has operator norm at most twice the largest edge modulus, the operator norms of the numerator and denominator matrices are at most $2M^{1/2}b_+^k$.

Choose $k_0$ so large that $2M^{1/2}b_+^k\leq1/2$ for $k\geq k_0$.  For any $m\times m$ matrix $B$ with $\|B\|\leq1/2$,
\begin{equation}
  |\det(I-B)|\leq \exp(2m\|B\|),
  \qquad
  |\det(I-B)|^{-1}\leq \exp(2m\|B\|).
\end{equation}
The first inequality follows from the singular-value bound on $\det(I-B)$, and the second from $s_{\min}(I-B)\geq1-\|B\|$.  Applying these estimates to the two determinant representations gives
\begin{equation}
  |D_k(q^{2k-1}px,p^2/q^2)|
  \leq \exp(C_1kb_+^k),
  \qquad k\geq k_0,
  \label{eq:Dk-tail-bound}
\end{equation}
with $C_1$ independent of $x,p,q$.

It remains to control the finitely many factors $1\leq k<k_0$.  For such a fixed $k$, Proposition~\ref{prop:weighted-path} applied with $Q=p^2/q^2>0$ shows that the zeros of $P_{k+1}(q^{2k-1}px,Q)$ in the $x$-plane are positive real numbers.  Denote them by $z_{k,j}(p,q)$, counted with multiplicity.  The smallest of these roots is the branch value corresponding to this denominator, and the remaining roots are larger.  Hence
\begin{equation}
  z_{k,j}(p,q)\geq x_k(p,q)\geq x_*(p,q).
\end{equation}
For the finitely many roots with $k<k_0$, continuity on the compact set $K$ also gives an upper bound $z_{k,j}(p,q)\leq Z<\infty$.  Therefore, whenever $|x|\leq x_*(p,q)-\epsilon$,
\begin{equation}
  |1-x/z_{k,j}(p,q)|
  \geq 1-\frac{|x|}{z_{k,j}(p,q)}
  \geq \frac{z_{k,j}(p,q)-|x|}{z_{k,j}(p,q)}
  \geq \frac{\epsilon}{Z}.
\end{equation}
Since the finite denominator polynomials have constant term one and factor as products of these terms, this gives a uniform lower bound away from zero.  Their numerator polynomials are continuous on the compact region $K\times\{|x|\leq M\}$, hence are uniformly bounded.  Thus the finitely many low-level $D_k$ factors are uniformly bounded on the same complex disks.

Combining the finite part with the tail estimate, and using
\begin{equation}
  \sum_{k\geq k_0}kb_+^k<\infty,
\end{equation}
shows that the product of all $D_k$-bounds appearing in \eqref{eq:Gh-D}, with the lower-level factors squared and the top factor single, is bounded uniformly in $h$.  Returning to \eqref{eq:Gh-D} gives
\begin{equation}
  |G_h(x,y,p,q)|\leq C_2 (MY)^h b_+^{h^2}.
\end{equation}
Because $b_+<1$, the quadratic factor $b_+^{h^2}$ dominates any fixed exponential growth coming from $(MY)^h$.  After enlarging the constant for finitely many small $h$, this is bounded by $C\theta^h$ for suitable $C<\infty$ and $0<\theta<1$.
\end{proof}

\begin{theorem}[Open-square radius theorem and branch selection]
\label{thm:open-square-radius}
For $0<p,q<1$, the radius of convergence of $G(x,1,p,q)$ as a power series in $x$ is
\begin{equation}
  R(p,q)=x_*(p,q)
  =\min_{k\geq1}\frac{1}{p^{2k-1}q\,\Lambda_{k+1}(q^2/p^2)}.
  \label{eq:two-weight-radius}
\end{equation}
For fixed $r>0$, set $\sigma=p^2$ and $q^2=r\sigma$.  As $\sigma$ varies in $0<\sigma<\min(1,r^{-1})$, the set of active branches can change only at neighbouring crossings, with possible ties.  If
\begin{equation}
  \sigma_k(r)=\frac{\Lambda_{k+1}(r)}{\Lambda_{k+2}(r)},
  \label{eq:two-weight-transition-values}
\end{equation}
then the $k$th adjacent crossing is parametrised by
\begin{equation}
  p^2=\sigma_k(r),\qquad q^2=r\sigma_k(r),
  \label{eq:two-weight-transition-curves}
\end{equation}
for those $r$ for which $\sigma_k(r)<\min\{1,r^{-1}\}$.  Several adjacent crossings may coincide, and the statement allows non-unique active branches.  The family of transition curves is symmetric under $p\leftrightarrow q$.
\end{theorem}

\begin{proof}
Fix $0<p,q<1$, and let $0<\delta<x_*(p,q)$.  By the continuity of $x_*$ on the open square, choose a compact neighbourhood $K\subset(0,1)^2$ of $(p,q)$ such that $x_*(p',q')>x_*(p,q)-\delta/2$ for all $(p',q')\in K$.  Proposition~\ref{prop:normal-convergence}, applied on $K$ with $Y=1$ and $\epsilon=\delta/2$, gives normal convergence uniformly for $|x|\leq x_*(p',q')-\delta/2$ and $(p',q')\in K$.  Since $|x|\leq x_*(p,q)-\delta$ lies inside these disks, the height expansion is normally convergent at $(p,q)$ on every closed disk $|x|\leq x_*(p,q)-\delta$.  Hence the sum is analytic for $|x|<x_*(p,q)$, and the radius of convergence is at least $x_*(p,q)$.

For the reverse inequality, choose an active index $k$, so that $x_k(p,q)=x_*(p,q)$.  If several indices attain the minimum, choose any one of them.  The top insertion factor
\begin{equation}
  D_k(q^{2k-1}px,p^2/q^2)
\end{equation}
is a bounded-height Dyck path generating function with positive weights.  Its Perron denominator zero is simple: the corresponding weighted path matrix is an irreducible Jacobi matrix with positive off-diagonal entries, and the bipartite pair $\pm\lambda$ contributes a single factor $1-X\lambda^2$ to the determinant in the variable $X$.  The numerator is the characteristic polynomial of the principal subpath obtained by deleting one endpoint of this path.  By strict Cauchy interlacing for irreducible Jacobi matrices \cite[Chapter~4]{hornjohnson2013}, the largest eigenvalue of this principal submatrix is strictly smaller than the Perron eigenvalue of the full matrix.  Hence the numerator is non-zero at the Perron denominator zero, so the Perron zero survives cancellation.  Thus this $D_k$ factor has radius $x_k(p,q)$ in the variable $x$.

All coefficients of the height expansion are non-negative when $p,q>0$ and $y=1$.  Moreover, the other factors in the height-$k$ summand have non-negative coefficients and constant term $1$.  Consequently the height-$k$ summand coefficientwise dominates a positive monomial times the above $D_k$ factor.  By the Cauchy--Hadamard formula, coefficientwise domination by a series of radius $x_k(p,q)$ implies that the height-$k$ summand has radius at most $x_k(p,q)$.  Since the full generating function $G(x,1,p,q)$ is the coefficientwise sum of the height contributions and dominates the height-$k$ contribution, again by Cauchy--Hadamard its radius is at most $x_k(p,q)=x_*(p,q)$.  This also covers branch crossings, because coefficientwise positivity precludes enlargement of the radius by cancellation among height contributions.  Combining both inequalities gives \eqref{eq:two-weight-radius}.

The branch-selection assertion follows from Corollary~\ref{cor:neighbouring-branches} applied to $\sigma=p^2$ and $r=q^2/p^2$; the hypotheses are exactly $p^2<1$ and $q^2=p^2r<1$.  The corollary allows a finite interval of maximising indices, so the theorem includes the case of ties.  The symmetry of the transition curves follows from the denominator symmetry \eqref{eq:R-symmetry}; equivalently, the denominator-zero identity gives the same branch locus after interchanging $p$ and $q$.
\end{proof}

The diagonal line $p=q=s$ is the explicit Chebyshev specialisation of Theorem~\ref{thm:open-square-radius}.  On this line the general denominator factors reduce to the Chebyshev factors of Section~\ref{sec:diagonal}, and the branch crossings occur at the explicit points \eqref{eq:sk}.  Away from the diagonal, the transition curves are obtained by solving equality conditions between neighbouring branches.  The first two are \eqref{eq:first-general-transition} and \eqref{eq:second-general-transition}.  Further curves are determined in the same way from the weighted-path Perron roots.

The first two transition curves, \eqref{eq:first-general-transition} and \eqref{eq:second-general-transition}, already show the symmetric neighbouring-crossing structure.  The branch selection is supplied by Corollary~\ref{cor:neighbouring-branches} and Theorem~\ref{thm:open-square-radius}, rather than by numerical plotting.

\section{Concluding remarks}
\label{sec:discussion}

We have obtained an exact joint area/water-capacity refinement of Dyck path enumeration.  The height decomposition expresses the generating function as an infinite sum of rational functions built from the finite polynomials $P_h(X,Q)$, combining the capacity decomposition with the height-restricted area machinery.  The representation separates the combinatorial construction from the analytic problem of selecting the radius-determining branch.

For the diagonal specialisation $p=q$, the model enumerates paths by the area below their upper hull.  The polynomial $P_h(X,1)$ reduces to a Chebyshev polynomial, and the positive denominator branches are explicit.  For $0<s<1$, the open-square theorem identifies the radius as their lower envelope.  The branch crossings accumulate at the ordinary Dyck critical point and produce the branch-envelope asymptotic curve \eqref{eq:xc-diagonal-asymp}.  A consistency check at the unweighted point is recorded in \ref{app:first-moment}.

Boundary regimes such as $p=1$, $q=1$, or weights outside the open square require separate analytic treatment.  The exact formulae and denominator symmetry remain valid there, but the compact tail estimate used in Theorem~\ref{thm:open-square-radius} is specific to the open square.

Within the open square $0<p,q<1$, the main analytic result is that the termwise denominator singularities determine the length radius.  The denominator symmetry gives a symmetric family of termwise denominator branch surfaces, while the weighted-path representation identifies the positive real branches with Perron roots of geometrically weighted path matrices.  The normal-convergence estimate rules out earlier singularities coming from the infinite height sum.  The interval log-submodularity theorem, whose proof is separated in \ref{app:interval-log-submodularity}, gives log-concavity of the Perron-root squares and thereby orders the adjacent crossings.  Hence the radius of convergence is the minimum in \eqref{eq:two-weight-radius}, and the selected branch, or set of selected branches at a tie, changes only when neighbouring branches agree.

The interval log-submodularity argument is independent of Dyck paths: it applies to spectral radii of overlapping intervals in any positively weighted path.  In the present problem, the geometric self-similarity of the edge weights specialises that interval inequality to the log-concavity used for branch selection.  This spectral comparison is then used in the proof of the open-square radius theorem.

\appendix

\section{First moment at the unweighted point}
\label{app:first-moment}

This appendix gives a consistency check independent of the proof of the two-weight radius theorem.  Set $y=1$, and let $\Dyck_n=\{\pi\in\Dyck:n(\pi)=n\}$.  On the diagonal $p=q=s$, the variable $s$ records hull area, so coefficientwise differentiation at $s=1$ gives
\begin{equation}
  [x^n]\left.\frac{\partial}{\partial s}G(x,1,s,s)\right|_{s=1}
  =\sum_{\pi\in\Dyck_n}\ell(\pi).
\end{equation}
Since $\ell(\pi)=a(\pi)+w(\pi)$, this moment is the sum of the ordinary proper-area moment and the water-capacity moment.  Thus, with
\begin{equation}
  A_n=\sum_{\pi\in\Dyck_n}a(\pi)=4^n-\binom{2n+1}{n}
\end{equation}
from the standard area enumeration \cite{chapman1999,woan2001}, and with $W_n=\sum_{\pi\in\Dyck_n}w(\pi)$ as given explicitly by Blecher, Brennan and Knopfmacher \cite[Theorem~4]{blecher2020}, one obtains
\begin{equation}
  [x^n]\left.\frac{\partial}{\partial s}G(x,1,s,s)\right|_{s=1}=A_n+W_n.
\end{equation}
This records only a consistency check: the diagonal weight differentiates to a statistic already decomposed as the sum of two known first moments.

\section{Proof of interval log-submodularity}
\label{app:interval-log-submodularity}

\begin{proof}[Proof of Lemma~\ref{lem:interval-log-submodularity}]
For $n=2$ the right-hand side is zero.  Let $n\geq3$.  Then $\rho_{2,n}>0$, so the divisions by $\gamma$ below are legitimate.  Put
\begin{equation}
  \alpha=\rho_{1,n},
  \qquad
  \beta=\rho_{2,n+1},
  \qquad
  \gamma=\rho_{2,n}.
\end{equation}
By monotonicity of spectral radius under passing from an induced subpath to a larger induced subpath, $\alpha\geq\gamma$ and $\beta\geq\gamma$.  Let $u,v,w$ be positive Perron eigenvectors on $[1,n]$, $[2,n+1]$, and $[2,n]$, respectively, extended by zero at their Dirichlet boundaries:
\begin{equation}
  u_0=u_{n+1}=0,
  \qquad
  v_1=v_{n+2}=0,
  \qquad
  w_1=w_{n+1}=0.
\end{equation}
For $1\leq i\leq n$, set
\begin{equation}
  \mathcal W_i=a_i(u_iw_{i+1}-u_{i+1}w_i).
\end{equation}
The eigenvalue equations give, on the common interval $2\leq i\leq n$, the discrete Wronskian identity
\begin{equation}
  \mathcal W_i-\mathcal W_{i-1}=(\gamma-\alpha)u_iw_i.
  \label{eq:wronskian-uw}
\end{equation}
Since $\mathcal W_n=0$ and $\alpha\geq\gamma$, it follows that $\mathcal W_i\geq0$.  Hence $u_i/w_i$ is decreasing on $2,\ldots,n$.  Similarly, set
\begin{equation}
  \mathcal Z_i=a_i(v_iw_{i+1}-v_{i+1}w_i),
  \qquad 1\leq i\leq n.
\end{equation}
Then
\begin{equation}
  \mathcal Z_i-\mathcal Z_{i-1}=(\gamma-\beta)v_iw_i,
  \qquad 2\leq i\leq n,
  \label{eq:wronskian-vw}
\end{equation}
with $\mathcal Z_1=0$ and $\beta\geq\gamma$, so $\mathcal Z_i\leq0$ and $v_i/w_i$ is increasing on $2,\ldots,n$.

Define a positive vector $z$ on $[1,n+1]$ by
\begin{equation}
  z_1=\frac{u_1v_2}{w_2},
  \qquad
  z_i=\frac{u_iv_i}{w_i}\quad(2\leq i\leq n),
  \qquad
  z_{n+1}=\frac{u_nv_{n+1}}{w_n}.
\end{equation}
We show that, for the weighted adjacency matrix $A$ on $[1,n+1]$,
\begin{equation}
  \frac{(Az)_i}{z_i}\leq \frac{\alpha\beta}{\gamma}
  \qquad(1\leq i\leq n+1).
  \label{eq:collatz-bound}
\end{equation}
At the two endpoints these ratios are $\alpha$ and $\beta$, respectively, and both are at most $\alpha\beta/\gamma$.  At an interior point $3\leq i\leq n-1$, put
\begin{equation}
\begin{array}{lll}
  \alpha_-=a_{i-1}u_{i-1}/u_i,&
  \alpha_+=a_i u_{i+1}/u_i,\\[3pt]
  \beta_-=a_{i-1}v_{i-1}/v_i,&
  \beta_+=a_i v_{i+1}/v_i,\\[3pt]
  \gamma_-=a_{i-1}w_{i-1}/w_i,&
  \gamma_+=a_i w_{i+1}/w_i.
\end{array}
\end{equation}
Then $\alpha_-+\alpha_+=\alpha$, $\beta_-+\beta_+=\beta$, and $\gamma_-+\gamma_+=\gamma$.  Moreover
\begin{equation}
  \frac{(Az)_i}{z_i}=\frac{\alpha_-\beta_-}{\gamma_-}
  +\frac{\alpha_+\beta_+}{\gamma_+}.
\end{equation}
The monotonicity of $u/w$ gives
\begin{equation}
  \alpha_-\gamma_+-\alpha_+\gamma_-\geq0,
\end{equation}
and the monotonicity of $v/w$ gives
\begin{equation}
  \beta_+\gamma_- - \beta_-\gamma_+\geq0.
\end{equation}
Therefore
\begin{align}
  &\frac{\alpha\beta}{\gamma}
  -\left(\frac{\alpha_-\beta_-}{\gamma_-}
  +\frac{\alpha_+\beta_+}{\gamma_+}\right)\nonumber\\
  &\qquad =\frac{(\alpha_-\gamma_+-\alpha_+\gamma_-)
  (\beta_+\gamma_- - \beta_-\gamma_+)}
  {\gamma_-\gamma_+(\gamma_-+\gamma_+)}\geq0.
\end{align}
The points $i=2$ and $i=n$ are boundary cases.  At $i=2$, define
\begin{equation}
  \alpha_- = a_1u_1/u_2,
  \qquad
  \alpha_+ = a_2u_3/u_2,
\end{equation}
so that $\alpha_-+\alpha_+=\alpha$.  One obtains
\begin{equation}
  \frac{(Az)_2}{z_2}=\alpha_-+\alpha_+\frac{\beta}{\gamma}
  \leq (\alpha_-+\alpha_+)\frac{\beta}{\gamma}
  =\frac{\alpha\beta}{\gamma},
\end{equation}
and the case $i=n$ is symmetric.  Thus \eqref{eq:collatz-bound} holds.  The Collatz--Wielandt formula \cite[Chapter~8]{hornjohnson2013} gives
\begin{equation}
  \rho_{1,n+1}\leq \frac{\alpha\beta}{\gamma},
\end{equation}
which is \eqref{eq:interval-log-submodularity}.
\end{proof}

\section{The second off-diagonal transition curve}
\label{app:second-transition}

This appendix records the short calculation leading to \eqref{eq:second-general-transition}.  Put $r=q^2/p^2$.  For the second and third branches,
\[
  \Lambda_3(r)=1+r,
\]
while $\Lambda_4(r)$ is the larger root of
\[
  L^2-(1+r+r^2)L+r^2=0.
\]
The equality $x_2(p,q)=x_3(p,q)$ is equivalent to
\[
  \frac{1+r}{p^2}=\Lambda_4(r).
\]
Substituting $r=q^2/p^2$ into the quadratic and multiplying by $p^8$ gives \eqref{eq:second-general-transition}.  In the open square, $p^2<1$ implies $(1+r)/p^2>1+r$, so the smaller root of the quadratic is excluded.  Thus the relevant open-square component is equivalently parametrised by
\[
  p^2=\frac{1+r}{\Lambda_4(r)},\qquad
  q^2=\frac{r(1+r)}{\Lambda_4(r)},\qquad r>0.
\]

\section*{Data availability}
No research data were generated or analysed during this theoretical study.

\section*{Declaration of generative AI and AI-assisted technologies in the manuscript preparation process}
During the preparation of this work the authors used OpenAI's ChatGPT to assist with language editing, consistency checks, and preparation of submission materials.  After using this tool, the authors reviewed and edited the content as needed and take full responsibility for the content of the published article.

\end{document}